\documentclass[12pt,reqno]{amsart}

\usepackage{amsmath}
\usepackage{amsfonts}
\usepackage{amssymb}
\usepackage{latexsym}   
\usepackage{graphicx,enumerate}
\usepackage{color}
\usepackage{hyperref,subfig}
\usepackage{algorithm,algorithmic}

\def\red{}
\def\rred{}

\usepackage{vmargin,fancyhdr}   
\usepackage{srcltx}
\def\cR{\mathcal{R}}
\newcommand{\diam}{\frac{ \log{n}}{\log{\log{n}}}}

\newcommand{\whp}{\textit{whp}}
\newcommand{\beql}[1]{\begin{equation}\label{#1}}
\newcommand{\eeq}{\end{equation}}

\newcommand{\Prob}[1]{{{\bf{Pr}}\left[{#1}\right]}}

\newtheorem{lemma}{Lemma}
\newtheorem{theorem}{Theorem}

\newtheorem{remark}{Remark}

\setpapersize{USletter}
\setmarginsrb{.75in}{.5in}        
             {.75in}{.5in}        
             {.25in}{.25in}     
             {.25in}{.5in}      
\newcommand{\hide}[1]{}

\newcommand{\field}[1]{\mathbb{#1}} 

\def\hT{\widehat{T}}
\def\hD{\widehat{D}}

\def\hC{\widehat{Q}}
\def\cC{\mathcal{C}}

\newcommand{\beq}[1]{\begin{equation}\label{#1}}

\newcounter{rot}

\def\a{\alpha}   \def\D{\Delta}
\def\e{\epsilon} \def\f{\phi}   \def\g{\gamma}
\def\G{\Gamma}  \def\k{\kappa}
 \def\th{\theta}    
   
\def\r{\rho}  \def\s{\sigma} 
\def\t{\tau} \def\om{\omega}

\newtheorem{corollary}[theorem]{Corollary}

\newcommand{\proofend}{\hspace*{\fill}\mbox{$\Box$}}

\newcommand{\rdup}[1]{{\left\lceil #1 \right\rceil }}
\newcommand{\rdown}[1]{{\lfloor #1 \rfloor}}

\newcommand{\brac}[1]{\left(#1\right)}

\newcommand{\bfrac}[2]{\left(\frac{#1}{#2}\right)}

\newcommand{\set}[1]{\left\{#1\right\}}

\def\Pr{\mbox{{\bf Pr}}}

\newcommand{\ignore}[1]{}

\def\hP{\widehat{P}}
\def\depth{\text{depth}}
\def\Bin{\text{Bin}}
\title{Rainbow Connection of Sparse Random Graphs}

\author{Alan Frieze}
\thanks{Alan Frieze's Research is supported in part by NSF Grant ccf1013110}

\author{Charalampos E. Tsourakakis}
\thanks{Charalampos E. Tsourakakis's Research is supported in part by NSF Grant ccf1013110}
\address{Department of Mathematical Sciences\\
Carnegie Mellon University\\
5000 Forbes Av., 15213\\
Pittsburgh, PA \\
U.S.A}
\email{alan@random.math.cmu.edu,ctsourak@math.cmu.edu}

\begin{document}

\begin{abstract} 
An edge colored graph $G$ is rainbow edge connected if any two vertices are connected
by a path whose edges have distinct colors. The rainbow connectivity of a connected
graph $G$, denoted by $rc(G)$, is the smallest number of colors that are needed in
order to make $G$ rainbow connected. 

In this work we study the rainbow connectivity of binomial random graphs
at the connectivity threshold $p=\frac{\log n+\om}{n}$ where $\om=\om(n)\to\infty$ and ${\om}=o(\log{n})$
and of random $r$-regular graphs where $r \geq 3$ is a fixed integer.
Specifically, we prove that the rainbow connectivity $rc(G)$ of $G=G(n,p)$ satisfies 
$rc(G) \sim \max\set{Z_1,diameter(G)}$
with high probability (\whp). Here $Z_1$ is the number of vertices in $G$ whose degree equals 1 and the 
diameter of $G$
is asymptotically equal to $\diam$ \whp.
Finally, we prove that the rainbow connectivity $rc(G)$ of the random $r$-regular graph $G=G(n,r)$ \whp\ satisfies 
$rc(G) =O(\log^{\th_r}{n})$ where $\th_r=\frac{\log (r-1)}{\log (r-2)}$ when $r\geq 4$ and $rc(G) =O(\log^4n)$ \whp\ when 
$r=3$.  

\end{abstract} 

\maketitle

\section{Introduction} 

Connectivity is a fundamental graph theoretic property. Recently, the  
concept of {\em rainbow connectivity} was introduced by Chartrand et al. in \cite{chartrand}. 
An edge colored graph $G$ is rainbow edge connected if any two vertices are connected by a 
path whose edges have distinct colors. The rainbow connectivity $rc(G)$ of a connected graph 
$G$ is the smallest number of colors that are needed in order to make $G$ rainbow edge connected. 
Notice, that by definition a rainbow edge connected graph is also connected and furthermore
any connected graph has a trivial edge coloring that makes it rainbow edge connected, since 
one may color the edges of a given spanning tree with distinct colors. 
Other basic facts established in \cite{chartrand} are that $rc(G)=1$ if and only if $G$ is a 
clique and $rc(G)=|V(G)|-1$ if and only if $G$ is a tree. 
Besides its theoretical interest, rainbow connectivity is also of interest in applied settings, such 
as securing sensitive information \cite{lisun}, transfer and networking \cite{chakraborty}.

The concept of  rainbow connectivity has attracted the interest of various researchers.
Chartrand et al. \cite{chartrand} determine the rainbow connectivity of several special classes
of graphs, including multipartite graphs. Caro et al. \cite{caro} prove that for a connected 
graph $G$ with $n$ vertices and minimum degree $\delta$, the rainbow connectivity satisfies
$rc(G)\leq \frac{\log{\delta}}{\delta}n(1+f(\delta))$, where $f(\delta)$ tends to zero 
as $\delta$ increases. The following simpler bound was also proved in \cite{caro},
$rc(G) \leq n \frac{4\log{n}+3}{\delta}$. 
{\red Krivelevich and Yuster \cite{krivelevichyuster} removed the logarithmic factor
from the Caro et al. \cite{caro} upper bound. Specifically they proved that
$rc(G) \leq \frac{20n}{\delta}$. 
Due to a construction of a graph
with minimum degree $\delta$ and diameter $\frac{3n}{\delta+1}-\frac{\delta+7}{\delta+1}$ by 
Caro et al. \cite{caro}, the best upper bound one can hope for is $rc(G) \leq \frac{3n}{\delta}$.
Chandran, Das, Rajendraprasad and Varma \cite{Chandran} have subsequently
proved an upper bound of $\frac{3n}{\delta+1}+3$, which is therefore essentially optimal.}

As Caro et al. point out, the random graph setting poses several intriguing questions. 
Specifically, let $G=G(n,p)$ denote the binomial random graph on $n$ 
vertices with edge probability $p$ \cite{erdosrenyi}. 
Caro et al. \cite{caro} proved that $p=\sqrt{\log{n}/n}$ is the sharp threshold for 
the property $rc(G(n,p))\leq 2$. 
He and Liang \cite{heliang} studied further the rainbow connectivity of random graphs. 
Specifically, they obtain the sharp threshold for the property $rc(G) \leq d$
where $d$ is constant. For further results and references we refer 
the interested reader to the recent survey of Li and Sun \cite{lisun}. 
In this work we look at the rainbow connectivity of the binomial graph
at the connectivity threshold  $p=\frac{\log{n}+{\om}}{n}$ where ${\om}=o(\log{n})$. 
This range of values for $p$ poses problems that cannot be tackled 
with the techniques developed in the aforementioned work. 
Rainbow connectivity has not been studied in random regular graphs to the best of our knowledge. 

Let 
\beq{Ldef}
L=\diam
\eeq
and let $A\sim B$ denote $A=(1+o(1))B$ as $n\to\infty$.

\noindent We establish the following theorems:
\begin{theorem}
\label{thrm:mainthrm}
Let $G = G(n,p),p=\frac{\log{n}+{\om}}{n}$, $\om\to\infty,{\om}=o(\log{n})$. Also, let $Z_1$ be the number of 
vertices of 
degree 1 in $G$.
Then, with high probability({\it whp})\footnote{An event $A_n$ holds with high probability ({\it whp}) 
if $\lim_{n \rightarrow +\infty} \Prob{A_n}=1$.}
$$  rc(G) \sim \max\set{Z_1, L},$$
\end{theorem} 
It is known that \whp\ the diameter of $G(n,p)$ is asymptotic to $L$ for $p$ as in the above range,
see for example Theorem 10.17 of Bollob\'as \cite{bollobas}. 
Theorem \ref{thrm:mainthrm} gives asymptotically optimal results. Our next theorem is not quite as precise.
\begin{theorem}
 \label{thrm:regular} 
Let $G=G(n,r)$ be a random $r$-regular graph where $r \geq 3$ is a fixed integer. 
Then, {\it whp} 
$$ rc(G) =\begin{cases}
          O(\log^4n)&r=3\\
          O(\log^{2\th_r}n)&r\geq 4.
          \end{cases}
$$ 
where $\th_r=\frac{\log (r-1)}{\log (r-2)}$.
\end{theorem}

\noindent All logarithms whose base is omitted are natural.
It will be clear from our proofs that the colorings in the above two theorems can be
constructed in a low order polynomial time. The second theorem, while weaker, contains
an unexpected use of a Markov Chain Monte-Carlo (MCMC) algorithm for randomly coloring a graph.

The paper is organized as follows: After giving a sketch of our approach in Section \ref{sketch},
in Sections~\ref{sec:proofs},~\ref{sec:regular} we prove 
Theorems~\ref{thrm:mainthrm},~\ref{thrm:regular}
respectively. Finally, in Section~\ref{sec:concl} we conclude by suggesting open problems.
\section{Sketch of approach}\label{sketch}
The general idea in the proofs of both theorems is as follows:
\begin{enumerate}[{\bf (a)}]
\item Randomly color the edges of the graph in question. For Theorem \ref{thrm:mainthrm} we can (in the main)
use a uniformly
random coloring. The distribution for Theorem \ref{thrm:regular} is a little more complicated. 
\item To prove that this works, we have to find, for each pair of vertices $x,y$, a large collection of edge disjoint
paths joining them. It will then be easy to argue that at least one of these paths is rainbow colored.
\item To find these paths we pick a typical vertex $x$. We grow a regular tree $T_x$ with root $x$. The depth is chosen
carefully. We argue that for a typical pair of vertices $x,y$, many of the leaves of $T_x$ and $T_y$ can be put
into 1-1 correspondence $f$ so that (i) the path $P_x$ from $x$ to leaf $v$ of $T_x$ is rainbow colored, (ii) the path 
$P_y$ from $y$ to the leaf $f(v)$ of $T_y$ is ranbow colored and (iii) $P_x,P_y$ do not share color.
\item We argue that from most of the leaves of $T_x,T_y$ we can grow a tree of depth approximately equal to half the diameter.
These latter trees themselves contain a bit more than $n^{1/2}$ leaves. These can be constructed so that they are vertex disjoint. Now
we argue that each pair of trees, one associated with $x$ and one associated with $y$, are joined by an edge. 
\item We now have, by construction, a large set of edge disjoint paths joining leaves $v$ of $T_x$ to leaves
$f(v)$ of $T_y$. A simple estimation shows that \whp\ for at least one leaf $v$ of $T_x$, the path from $v$ to $f(v)$ is 
rainbow colored and does not
use a color already used in the path from $x$ to $v$ in $T_x$ or the path from $y$ to $f(v)$ in $T_y$.
\end{enumerate}
We now fill in the details of both cases. 
\section{Proof of Theorem~\ref{thrm:mainthrm}} 
\label{sec:proofs}

Observe first that $rc(G)\geq \max\set{Z_1,diameter(G)}$. First of all, each edge incident to a vertex of 
degree one must have a distinct color.
Just consider a path joining two such vertices. Secondly, if the shortest distance between two vertices is $\ell$ 
then we need at least $\ell$ colors.
Next observe that \whp\ the diameter $D$ is asymptotically equal to $L$, 
see for example \cite{bollobas}. 
We break the proof of Theorem~\ref{thrm:mainthrm} into several lemmas. 

Let a vertex be {\em large} if $\deg(x)\geq \log{n}/100$ and {\em small} otherwise.
\begin{lemma} 
\label{lem2}
{\em Whp}, there do not exist two small vertices within distance 
at most $3L/4$. 
\end{lemma} 
\begin{proof}
 \begin{align*}
&\Prob{\exists x,y \in [n]:\;\deg(x),\deg(y) \leq \log{n}/100\text{ and }dist(x,y) \leq \frac{3L}{4}}\\
&\leq \binom{n}{2} \sum_{k=1}^{3L/4} n^{k-1}p^k
\brac{\sum_{i=0}^{\log n/100}\binom{n-1-k}{i}p^i(1-p)^{n-1-k}}^2\\
&\leq \sum_{k=1}^{3L/4} n(2\log n)^k\brac{2\binom{n}{\log n/100}p^{\log n/100}(1-p)^{n-1-\log n/100}}^2\\
&\leq \sum_{k=1}^{3L/4} n(2\log n)^k\brac{2(100e^{1+o(1)})^{\log n/100}n^{-1+o(1)}}^2\\
&\leq \sum_{k=1}^{3L/4} n(2\log n)^k n^{-1.9}\\
&\leq 2n(2\log n)^{3L/4}n^{-1.9}\\
&\leq n^{-.1}.
\end{align*}

\end{proof}

\noindent 
We use the notation $e[S]$ for the number of edges induced by a given set of vertices $S$. Notice that 
if a set $S$ satisfies $e[S] \geq s+t$ where $t\geq 1$, the induced subgraph $G[S]$ has at least $t+1$ cycles.

\begin{lemma}
\label{lem3} 
Fix $t\in \field{Z}^+$ and $0<\alpha<1$. Then, {\em whp} there does not exist a subset $S \subseteq [n]$, 
such that $|S| \leq \alpha t L$ and $e[S] \geq |S|+t$. 
\end{lemma}

\begin{proof}

For convenience, let $s=|S|$ be the cardinality of the set $S$.Then,
\begin{align*}
\Prob{ \exists S: s \leq \alpha t L \text{~~and~~} e[S] \geq s+t } 
&\leq \sum_{s \leq \alpha t L} \binom{n}{s} \binom{\binom{s}{2}}{s+t} p^{s+t}\\
&\leq \sum_{s \leq \alpha t L } \bfrac{ne}{s}^s \bfrac{es^2p}{2(s+t)}^{s+t} \\ 
&\leq \sum_{s \leq \alpha t L } (e^{2+o(1)}\log{n})^s \bfrac{es\log{n}}{n}^t \\
&\leq  \alpha tL \brac{ (e^{2+o(1)}\log{n})^{ \alpha L} \bfrac{e \alpha t\log^2{n} }{n\log{\log{n}}}}^t \\ 
&< \frac{1}{n^{(1-\a-o(1))t}}.
\end{align*}

\end{proof}

\begin{remark}\label{rem1}
Let $T$ be a rooted tree of depth at most $4L/7$ and let $v$ be a vertex not in $T$, but with $b$ neighbors in $T$.
Let $S$ consist of $v$, the neighbors of $v$ in $T$ plus the ancestors of these neighbors. Then 
$|S|\leq 4bL/7+1\le 3bL/5$ and $e(S)=|S|+b-2$. It follows from {\red the proof of} 
Lemma \ref{lem3} with $\a=3/5$ and $t=8$, that 
we must have $b\leq 10$
with probability $1-o(n^{-3})$.
\end{remark}

\noindent Our next lemma shows the existence of the subgraph $G'_{x,y}$ described next and shown in 
Figure~\ref{fig:fig1}
for a given pair of vertices $x,y$. 
We first deal with paths between large vertices.

\begin{figure*}
\includegraphics[width=0.7\textwidth]{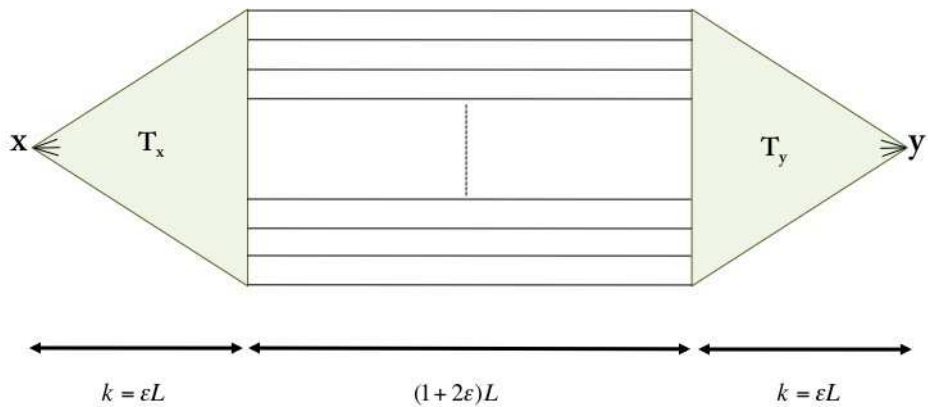}
\caption{Structure of Lemma~\ref{lem4}.}
\label{fig:fig1}
\end{figure*}
Now let 
\beq{eps}
\text{$\e=\e(n)=o(1)$ be such that $\frac{\e\log\log n}{\log 1/\e}\to \infty$ and let $k=\e L$.}
\eeq
Here $L$ is defined in \eqref{Ldef} and we could take $\e=1/(\log\log n)^{1/2}$.
\begin{lemma}
\label{lem4} 
{\red Whp,} for all pairs of large vertices $x,y \in [n]$
there exists a subgraph $G_{x,y}(V_{x,y},E_{x,y})$ of $G$ as shown in figure~\ref{fig:fig1}.
The subgraph consists of two isomorphic vertex disjoint trees $T_x,T_y$ rooted at $x,y$ each of depth $k$.
$T_x$ and $T_y$ both have a branching factor of $\log n/101$. I.e. each vertex of $T_x,T_y$ has
at least  $\log n/101$ neighbors, excluding its parent in the tree.
Let the leaves of $T_x$ be $x_1,x_2,\ldots,x_\t$ where $\t\geq n^{4\e/5}$ and those of $T_y$ be
$y_1,y_2,\ldots,y_\t$. Then $y_i=f(x_i)$ where $f$ is a natural 
isomporphism that preserves the parent-child relation.
Between each pair of leaves $(x_i,y_i),i=1,2,\ldots,\t$ 
there is a path $P_i$ of length  $(1+2\epsilon) L$. The paths $P_i,i=1,2,\ldots,\t$ are edge disjoint.
\end{lemma}

\begin{proof}
Because we have to do this for all pairs $x,y$, we note without further comment that likely (resp. unlikely) 
events will
be shown to occur with probability $1-o(n^{-2})$ (resp. $o(n^{-2}$)).

To find the subgraph shown in Figure~\ref{fig:fig1} we grow tree structures as shown 
in Figure~\ref{fig:fig2}. 
Specifically, we first grow a tree from $x$ using BFS until it reaches  depth $k$. 
Then, we grow a tree starting from $y$ again using BFS until it reaches depth $k$. 
Finally, we grow trees from the leaves of $T_x$ and $T_y$ using BFS for depth $\gamma=(\frac{1}{2}+\epsilon)L$. 
Now we analyze these processes. Since the argument is the same we explain
it in detail for $T_x$ and we outline the differences for the other trees. 
We use the notation $D_i^{(\r)}$ for the number of vertices at depth $i$
of the BFS tree rooted at $\r$.

\begin{figure*}
\includegraphics[width=0.7\textwidth]{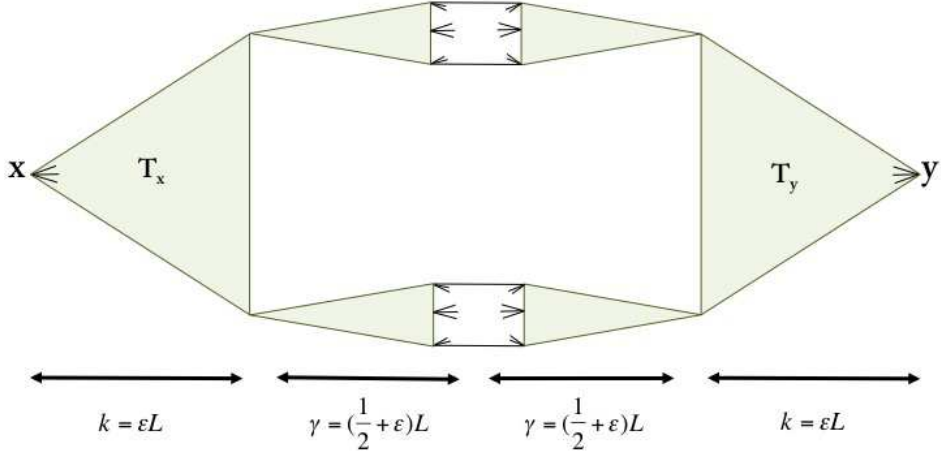}
\caption{Subgraph found in the proof of Lemma~\ref{lem4}.}
\label{fig:fig2}
\end{figure*}

First we grow $T_x$. As we grow the tree via BFS from a vertex $v$ at depth $i$ to vertices 
at depth $i+1$ certain {\em bad} edges from $v$ may point 
to vertices already in $T_x$. Remark \ref{rem1} shows with probability $1-o(n^{-3})$ there can be at most
10 bad edges emanating from $v$.

 Furthermore, 
Lemma \ref{lem2} implies that there exists at most one vertex of degree less than $\frac{\log{n}}{100}$ 
at each level {\it whp}. 
Hence, we obtain the recursion  

\beq{rec1}
D_{i+1}^{(x)} \geq \brac{\frac{\log{n}}{100}-10} (D_i^{(x)}-1) \geq \frac{\log{n}}{101}D_i^{(x)}.
\eeq

\noindent Therefore the number of leaves satisfies 

\beq{rec2}
D_{k}^{(x)} \geq \bfrac{\log n}{101}^{\e L}\geq n^{4\e/5}.
\eeq
We can make the branching factor exactly $\frac{\log n}{101}$ by pruning.
We do this so that the trees $T_x$ are isomorphic to each other.

With a similar argument 
\beq{rec3}
D_{k}^{(y)} \geq n^{\frac{4}{5}\epsilon}.
\eeq
The only difference is that now we also say an edge is bad if the other endpoint is in $T_x$.
This immediately gives
$$D_{i+1}^{(y)} \geq \brac{\frac{\log{n}}{100}-20} (D_i^{(y)}-1) \geq \frac{\log{n}}{101}D_i^{(y)}$$
and the required conclusion \eqref{rec3}.

Similarly, from each leaf $x_i \in T_x$ and $y_i \in T_y$ we grow trees $\hT_{x_i},\hT_{y_i}$ of depth 
$\gamma = \big(\frac{1}{2}+\epsilon\big) L$ using the same procedure and arguments 
as above. Remark \ref{rem1} implies that there are at most 20 edges from the vertex $v$ being explored to 
vertices in any of the trees already constructed. At most 10 to $T_x$ plus any trees rooted at an $x_i$ 
and another 10 for $y$.
The numbers of leaves of each $\hT_{x_i}$ now satisfies
$$\hD_\g^{(x_i)}\geq \frac{\log n}{100}\bfrac{\log n}{101}^{\g}\geq n^{\frac12+\frac{4}{5}\epsilon}.$$
Similarly for $\hD_\g^{(y_i)}$.

\noindent Observe next that BFS does not condition the edges between the leaves $X_i,Y_i$ of the trees $\hT_{x_i}$ 
and $\hT_{y_i}$.
I.e., we do not need to look at these edges in order to carry out our construction. 
On the other hand we have conditioned on the occurence of certain events to imply a certain growth rate.
We handle this technicality as follows. We go through the above construction and halt if ever we find that we cannot
expand by the required amount. Let ${\bf A}$ be the event that we do not halt the construction
i.e. we fail the conditions of Lemmas \ref{lem2} or \ref{lem3}. We have
$\Prob{{\bf A}}=1-o(1)$ and so,
$$\Prob{\exists i:e(X_i,Y_i)=0\mid {\bf A}}\leq \frac{\Prob{\exists i:e(X_i,Y_i)=0}}{\Pr({\bf A})}\leq
2n^{\frac{4\e}{5}}(1-p)^{n^{1+\frac{8\e}{5}}}\leq n^{-n^\e}.$$
We conclude that {\em whp} there is always an edge between each $X_i,Y_i$ and thus a path of length at most
$(1+2\e)L$ between each $x_i,y_i$.
\end{proof}

\noindent Let $q=(1+5\e)L$ be the number of available colors. 
We color the edges of $G$ randomly. 
We show that the probability of having a rainbow path between $x,y$ in the 
subgraph $G_{x,y}$ of Figure~\ref{fig:fig1}  is at least $1-\frac{1}{n^3}$. 

\begin{lemma}
\label{lem5}
Color each edge of $G$ using one color at random from $q$ available. Then, the probability of having at least one 
rainbow path 
between two fixed large vertices $x,y \in [n]$ is at least $1-\frac{1}{n^3}$. 
\end{lemma}

\begin{proof}

We show that the subgraph $G_{x,y}$ contains such a path.
We break our proof into two steps:
 
\begin{figure*}
\includegraphics[width=0.7\textwidth]{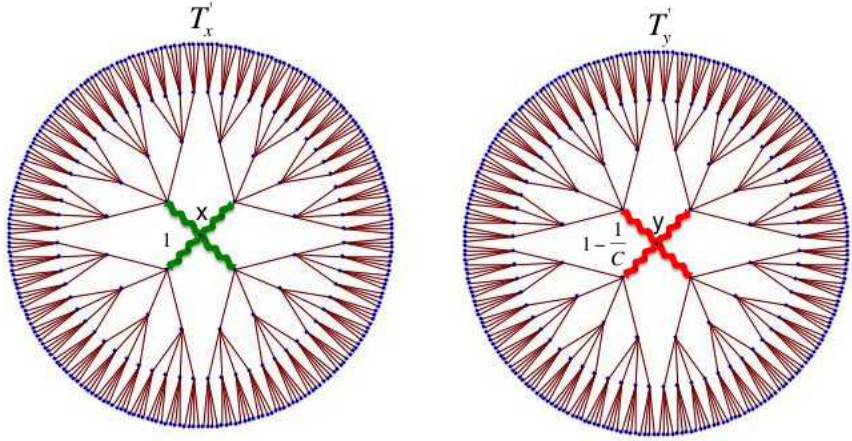}
\caption{Figure shows $\frac{\log{n}}{101}$-ary trees $T_x,T_y$. The two roots are shown respectively at the 
center of the trees. 
In our thinking of the random coloring as an evolutionary process, the green edges incident to $x$ survive
with probability 1, the red edges incident to $y$ with probability $1-\frac{1}{q}$ and all the other
edges with probability $p_0= \Big( 1- \frac{2k}{q} \Big)^2$ where $k$ is the depth of both trees
and $q$ the number of available colors. 
Our analysis in Lemma~\ref{lem4} using these probabilities gives a lower bound on the number of 
alive pairs of leaves after coloring $T_x,T_y$ from the root to the leaves respectively.}
\label{fig:fig3}
\end{figure*}

Before we proceed, we provide certain necessary definitions. 
Think of the process of coloring $T_x,T_y$ as an evolutionary process that colors edges 
by starting from the two roots $x, f(x)=y$ until it reaches the leaves. 
In the following, we call a vertex $u$ of $T_x$ ($T_y$)
{\red {\em alive/living}} if the path $P(x,u)$ ($P(y,u)$) from $x$ ($y$) to $u$ is rainbow, i.e., the edges have received 
distinct colors. 
We  call a pair of vertices $\{u,f(u)\}$ alive, $u \in T_x, f(u) \in T_y$ if $u,f(u)$ are both {\it alive}
and the paths $P(x,u),P(y,f(u))$ share no color. 
Define $A_j=|\{(u,f(u)): (u,f(u)) \text{~is alive and depth}(u)=j  \}|$ for $j=1,..,k$.

\noindent
\underline{$\bullet$ {\sc Step 1:} Existence of at least $n^{\tfrac{4}{5}\epsilon}$ living pairs of leaves}\\

Assume the pair of vertices $\{u,f(u)\}$ is alive where $u \in T_x, f(u) \in T_y$. 
It is worth noticing that $u,f(u)$
have the same depth in their trees. We are interested in the number of pairs of children 
$\{u_i, f(u_i)\}_{i=1,..,\log{n}/101}$ 
that will be alive after coloring the edges from $\depth(u)$ to $\depth(u)+1$. 
A living pair $\{u_i,f(u_i)\}$ by definition has the following properties:
edges $(u,u_i) \in E(T_x)$ and $(f(u),f(u_i))\in E(T_y)$ receive two distinct colors, 
which are different from the set of colors used in paths $P(x,u)$ and $P(y,f(u))$.  
Notice the latter set of colors has cardinality $2 \times \text{depth}(u) \leq 2 k$. 

Let $A_j$ be the number of living pairs at depth $j$. We first bound the size of $A_1$.
\beq{A1}
{\red \Prob{A_1\leq \frac{\log n}{200}}\leq 2^{\log n/101}\bfrac{1}{q}^{\log n/300}=O(n^{-\Omega(\log\log n)}).}
\eeq
{\red Here $2^{\log n/101}$ bounds the number of choices for $A_1$. For a fixed set $A_1$ there will
be at least $\frac{\log n}{101}-\frac{\log n}{200}\geq \frac{\log n}{300}$ edges incident with 
$x$ that have the same color as their corresponding edges incident with $y$, under $f$. The factor $q^{-\log n/300}$
bounds the probability of this event}. 

For $j>1$ we see that the random variable equal to the number of living 
pairs of children of $(u,f(u))$ stochastically dominates
the random variable $X \sim \Bin\brac{\frac{\log{n}}{101}, p_0}$, where $p_0 =  
\brac{1-\frac{2k}{q}}^2 = \big(\frac{1+3\epsilon}{1+5\epsilon}\big)^2$. 
The colorings of the descendants of each live pair are independent and so
we have using the Chernoff bounds for $2\leq j\leq k$,
\begin{multline}\label{indu}
\Prob{A_j < \bfrac{\log{n}}{200}^j p_0^{j-1} \bigg| A_{j-1} \geq \bfrac{\log{n}}{200}^{j-1} p_0^{j-2}} \\
\leq \exp\set{-\frac12\cdot\bfrac{99}{200}^2\cdot\frac{\log n}{101}\cdot\bfrac{\log n}{200}^{j-1}p_0^j}=
O(n^{-\Omega(\log\log n)}).
\end{multline}

\eqref{A1} and \eqref{indu} justify assuming that $A_k \geq \bfrac{\log n}{200}^k
{\red p_0^{k-1}}\ge n^{\frac{4}{5}\epsilon}$.

\noindent \\
\underline{$\bullet$ {\sc Step 2:} Existence of rainbow paths between $x,y$ in $G_{x,y}$}\\
Assuming that there are $\geq n^{4\e/5}$ living pairs of leaves $(x_i,y_i)$ for vertices $x,y$, 
$$\Pr(x,y\text{ are not rainbow connected})\leq \brac{1-\prod_{i=0}^{2\g-1}\brac{1-\frac{2k+i}{q}}}^{n^{4\e/5}}.$$
But
$$\prod_{i=0}^{2\g-1}\brac{1-\frac{2k+i}{q}}\geq \brac{1-\frac{2k+2\g}{q}}^{2\g}=\bfrac{\e}{1+5\e}^{2\g}.$$
So
\begin{multline}\label{lem10}
\Pr(x,y\text{ are not rainbow connected})\leq \exp\set{-n^{4\e/5}\bfrac{\e}{1+5\e}^{2\g}}\\
=\exp\set{-n^{4\e/5-O(\log(1/\e)/\log\log n)}}.
\end{multline}
Using \eqref{eps} and the union bound taking \eqref{lem10} over all large $x,y$ completes the proof of 
Lemma \ref{lem5}.
\end{proof}

\begin{figure}
  \centering
  \subfloat[]{\label{fig4:4a}\includegraphics[width=0.3\textwidth]{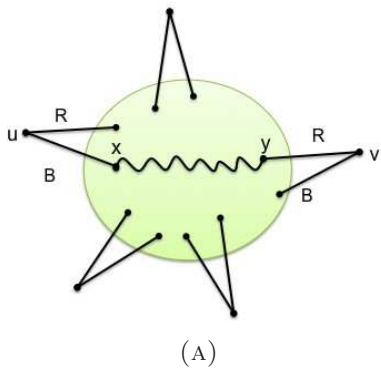}}                
  \caption{Taking care of small vertices.}
  \label{fig:fig4}
\end{figure}

\noindent We now finish the proof of Theorem~\ref{thrm:mainthrm} i.e. take care of small vertices. 

We showed in Lemma~\ref{lem5} that \whp\ for any two large vertices, a random coloring
results in a rainbow path joining them. We divide
the small vertices into two sets: vertices of degree 1, $V_1$ and the
vertices of degree at least 2, $V_2$. Suppose that our colors are $1,2,\ldots,q$
and $V_1=\set{v_1,v_2,\ldots,v_s}$. We begin by giving the edge incident with $v_i$ the color $i$.
Then we slightly modify the argument in Lemma \ref{lem5}. If $x$ is the neighbor
of $v_i\in V_1$ then color $i$ cannot be used in Steps 1 and 2 of that procedure. In terms of {\red analysis} this 
replaces $q$ by $(q-1)$ ($(q-2)$ if $y$ is also a neighbor of $V_1$) and the argument is essentially unchanged
i.e. \whp\ there will be a rainbow path between each pair of large vertices.  
Furthermore, any path starting at $v_i$ can only use color $i$ once and so there will be rainbow paths between
$V_1$ and $V_1$ and between $V_1$ and the set of large vertices.

The set $V_2$ is treated by using only two extra colors. 
Assume that Red and Blue have not been used in our coloring. Then we use
Red and Blue to color two of the edges incident to a vertex $u\in V_2$ (the remaining edges are colored arbitrarily).
This is shown in Figure~\ref{fig4:4a}. Suppose that $V_2=\set{w_1,w_2,\ldots,w_t}$. Then if we want a rainbow path 
joining $w_i,w_j$ where $i<j$ then we use the red edge to go to its neighbor $w_i'$. Then we take the already  
constructed rainbow path to $w_j''$, the neighbor of $w_j$ via a blue edge. Then we can continue to $w_j$.
\proofend

\section{Proof of Theorem~\ref{thrm:regular}} 
\label{sec:regular} 
We first observe that simply randomly coloring the edges of $G=G(n,r)$ with $q=n^{o(1)}$ colors will 
not do. This is because there will \whp\ be {\red $\Omega(nq^{1-r^2})=\Omega(n^{1-o(1)})$} vertices $v$ where all 
edges at distance
at most two from $v$ have the same color.

We follow a similar 
strategy to the proof in Theorem \ref{thrm:mainthrm}. We grow small trees $T_x$ from each vertex $x$. 
Then for a pair
of vertices $x,y$ we build disjoint trees on the 
leaves of $T_x,T_y$ so that \whp\ we can find edge disjoint paths between any set of leaves $S_x$ of 
$T_x$ and any set of leaves of $S_y$ of the same size. A bounded number of leaves of $T_x,T_y$
will be excluded from this statement. The main difference will
come from our procedure for coloring the edges. Because of the similarities, we will give a little less detail in
the common parts of our proofs. We are in effect
talking about building a structure like that shown in Figure \ref{fig:fig2}. There is one difference, we will have
to take care of which leaves of $T_x$ we pair with which leaves of $T_y$, for a pair of vertices $x,y$.

Having grown the trees, we have the problem of coloring the edges.
Instead of independently and randomly coloring the edges, we use a greedy algorithm that produces
a coloring that is guaranteed to color edges differently, if they are close. This will guarantee that 
the edges of $T_x$ are rainbow, for all vertices $x$. We then argue that we can find, for each vertex pair
$x,y,$ a partial mapping $g$ from the leaves of $T_x$ to the leaves of $T_y$ such that the path from $x$ to leaf
$v$ in $T_x$ and the path from $y$ to leaf $g(v)$ in $T_y$ do not share a color. This assumes that $v$ has an image under
the partial mapping $g$. We will have to argue that $g$ is defined on enough vertices in $T_x$. Given this,
we then consider the colors on a set of edge disjoint paths that we can construct from the leaves of $T_x$ to their 
$g$-counterpart in the leaves of $T_y$.

{\red We will use the configuration model of Bollob\'as \cite{b1} in our proofs, see \cite{JLR} 
or \cite{wormald} for details.}
Let $W=[2m=rn]$ be our set
of {\em configuration points} and let $W_i=[(i-1)r+1,ir]$,
$i\in [n]$, partition $W$. The function $\f:W\to[n]$ is defined by
$w\in W_{\f(w)}$. Given a
pairing $F$ (i.e. a partition of $W$ into $m$ pairs) we obtain a
(multi-)graph $G_F$ with vertex set $[n]$ and an edge $(\f(u),\f(v))$ for each
$\{u,v\}\in F$. Choosing a pairing $F$ uniformly at random from
among all possible pairings $\Omega_W$ of the points of $W$ produces a random
(multi-)graph $G_F$. 
{\red Each $r$-regular simple graph $G$ on vertex set $[n]$ is equally likely to be generated as $G_F$.
Here simple means without loops of multiple edges. 
Furthermore, if $r=O(1)$ then $G_F$ is simple with a probability bounded below by a positive value independent of $n$.
Therefore, any event that occurs \whp\ in $G_F$ will also occur \whp\ in $G(n,r)$.}
\subsection{Tree building}
We will grow a Breadth First Search tree $T_x$ from each vertex. We will grow each tree to depth 
$$k=k_r=\begin{cases}\rdup{\log_{r-2}\log n}&r\geq 4.\\\rdup{2\log_2\log n-2\log_2\log_2\log n}&r=3.\end{cases}$$
Observe that 
\beq{neq3}
T_x\text{ has at most }r(1+(r-1)+(r-1)^2+\cdots+(r-1)^{k-1})=r\frac{(r-1)^k-1}{r-1}\text{ edges.}
\eeq
It is useful to observe that
\begin{lemma}\label{density}
Whp, no set of $s\leq \ell_1=\frac{1}{10}\log_{r-1}n$ vertices contains more than $s$ edges.
\end{lemma}
\begin{proof} 
Indeed,
\begin{align}
\Pr(\exists S\subseteq [n],|S|\leq \ell_1,e[S]\geq |S|+1)&\leq 
\sum_{s=3}^{\ell_1}\binom{n}{s}\binom{\binom{s}{2}}{s+1}\bfrac{r^2}{rn-rs}^{s+1}\label{neq1}\\
&\leq \frac{r\ell_1}{n}
\sum_{s=3}^{\ell_1}\binom{n}{s}\binom{\binom{s}{2}}{s}\bfrac{r^2}{rn-rs}^{s}\nonumber\\
&\leq \frac{r\ell_1}{n}\sum_{s=3}^{\ell_1}\brac{\frac{ne}{s}\cdot \frac{se}{2}\cdot\frac{2r}{n}}^s\nonumber\\
&\leq \frac{r\ell_1}{n}\cdot\ell_1\cdot (e^2r)^{\ell_1}=o(1).\label{neq2}
\end{align}
{\bf Explanation of \eqref{neq1}:} The factor $\bfrac{r^2}{rn-rs}^{s+1}$ can be justified as follows. We can estimate
$$\Pr(e_1,e_2,\ldots,e_{s+1}\in E(G_F))=\\\prod_{i=0}^{s}\Pr(e_{i+1}\in E(G_F)\mid e_1,e_2,\ldots,e_i\in E(G_F))
\leq \bfrac{r^2}{rn-rs}^{s+1}$$
if we pair up the lowest index endpoint of each $e_i$ in some arbitrary order. The fraction $\frac{r^2}{rn-rs}$ 
is an upper bound
on the probability that this endpoint is paired with the other endpoint, regardless of previous pairings.
\end{proof}

Denote the leaves of $T_x$ by $L_x$.
\begin{corollary}\label{nlem1}
Whp, {\red $(r-1)^k\leq |L_x|\leq r(r-1)^{k-1}$ for all $x\in [n]$.}
\end{corollary}
\begin{proof}
This follows from the fact that \whp\ the vertices spanned by each $T_x$ span at most one cycle. 
This in turn follows from Lemma \ref{density}.
\end{proof}

\noindent Consider two vertices $x,y \in V(G)$ where $T_x\cap T_y=\emptyset$. We will show that \whp\ we can find 
a subgraph $G'(V',E'), V' \subseteq V, E'\subseteq E$ with 
similar structure to that shown in Figure~\ref{fig:fig2}. Here $k=k_r$ and $\gamma=\brac{\frac12+\e}\log_{r-1}n$
for some small positive constant $\e$. 
\begin{remark}\label{rem3}
In our analysis we expose the pairing $F$, only as necessary. For example
the construction of $T_x$ involves exposing all pairings involving non-leaves of $T_x$ and one pairing for each leaf.
There can be at most one exception to this statement, for the rare case where $T_x$ contains a unique cycle.
In particular, if we expose the point $q$ paired with a currently unpaired point $p$ of a leaf of $T_x$ then $q$ is chosen
randomly from the remaining unpaired points.
\end{remark}
Suppose that we have 
constructed $i=O(\log n)$ {\red vertex disjoint trees of depth $\g$ rooted at some of the leaves 
of $T_x$}. We grow the $(i+1)$st tree $\hT_z$ via BFS, without using edges that go into $y$ or previously 
constructed trees. 
Let a leaf $z\in L_x$ be {\em bad} if we have to omit a single edge as we construct the first $\ell_1/2$ levels 
of $\hT_z$.
The previously constructed
trees plus $y$ account for $O(n^{1/2+\e})$ vertices and pairings, so the probability that $z$ is bad, given all
the pairings we have exposed so far, is at most 
$O((r-1)^{\ell_1/2}n^{-1/2+\e})=O(n^{-1/3})$. Here bad edges can only join two leaves. This probability bound holds
regardless of whichever other vertices are bad. This follows from the way we build the pairing $F$, see the final
statement of Remark \ref{rem3}. So \whp\ there will be at most 3 bad leaves on any 
$T_x$. Indeed, $\Pr(\exists x:x\text{ has }\geq 4
\text{ bad leaves})\leq n\binom{O(\log n)}{4}n^{-4/3}=o(1)$. 

If a leaf is not bad then the first $\ell_1/2$ 
levels produce $\Theta(n^{1/20})$ leaves. 
From this, we see that \whp\ the next $\g-\ell_1$ levels grow at a rate $r-1-o(n^{-1/25})$. Indeed, given that a 
level has $L$ vertices where 
$n^{1/20}\leq L\leq n^{3/4}$, the number of vertices in the next level dominates 
$Bin\brac{(r-1)L,1-O\bfrac{n^{3/4}}{n}}$,
after accounting for the configuration points used in building previous trees. 
Indeed, $(r-1)L$ configuration points associated with good leaves will be unpaired and for each of them,
the probability it is paired with a point associated with a vertex in any of the trees constructed so far is
$O(n^{1/2+2\e}/n)$. This probability bound holds regardless of the pairings of the other leaf configuration points.  
We can thus assert that \whp\ we will have that all but at most three of the leaves $L_x$ of
$T_x$ are roots of vertex disjoint trees $\hT_1,\hT_2,\ldots,$ 
each with $\Theta(n^{1/2+\e/2})$ leaves. Let $L_x^*$ denote these {\em good} leaves.
The same analysis applies when we build trees $\hT_1',\hT_2',\ldots,$ with roots at $L_y$. 

Now the probability that there is no edge joining the leaves of $\hT_i$ to the leaves of $\hT_j'$ is at most
$$\brac{1-\frac{(r-1)\Theta(n^{1/2+\e/2})}{rn}}^{(r-1)
n^{1/2+\e/2}}\leq e^{-\Omega(n^{\e})}.$$ 
To summarise,
\begin{remark}\label{rem2}
{\em Whp} we will succeed in finding in $G_F$ and hence in $G=G(n,r)$, for all $x,y\in V(G_F)$,
for all $u\in L_x^*,v\in L_y^*$, a path $P_{u,v}$ from $u$ to $v$ of length
$O(\log n)$ such that if $u\neq u'$ and $v\neq v'$ then
$P_{u,v}$ and $P_{u',v'}$ are edge disjoint. These paths avoid $T_x,T_y$ except at their start and endpoints.
\end{remark}
\subsection{Coloring the edges}
We now consider the problem of coloring the edges of $G$. Let $H$ denote the line graph of 
$G$ and let $\G=H^{2k}$ denote the 
graph with the same vertex set as $H$ and an edge between
vertices $e,f$ of $\G$ if there there is a path of length at most $k$ between $e$ and $f$ in $H$. We will
construct a proper coloring
of $\G$ using 
$$q=10(r-1)^{2k}\sim100\log^{2\th_r}n\text{ where }\th_r=\frac{\log (r-1)}{\log (r-2)}$$ 
colors. We do this as follows: Let $e_1,e_2,\ldots,e_m$ be an 
arbitrary ordering of the vertices of $\G$. For $i=1,2,\ldots,m$, color $e_i$ with a random color, chosen 
uniformly from the set of colors not currrently
appearing on any neighbor in $\G$. At this point only $e_1,e_2,\ldots,e_{i-1}$ will have been colored.

Suppose then that we color the edges of $G$ using the above method. Fix a pair of vertices $x,y$ of $G$.
We see immediately, that no color appears twice in $T_x$ and no color appears twice in $T_y$. 
This is because the distance between edges in $T_x$ is at most $2k$. 
This also deals with the case where $V(T_x)\cap V(T_y)\neq\emptyset$, for the same reason. So assume now that 
$T_x,T_y$ are vertex disjoint.
We can find lots of paths joining $x$ and $y$. We know that the first and last $k$ edges of each
path will be individually rainbow colored. We will first show that we have many choices of path where these $2k$ edges 
are rainbow colored
when taken together. 

\subsection{Case 1: $r\geq 4$:}
We argue now that we can  
find $\s_0=(r-2)^{k-1}$ leaves 
$u_1,u_2,\ldots,u_\t\in T_x$ and $\s_0$ leaves $v_1,v_2,\ldots,v_\t\in T_y$
such for each $i$ the $T_x$ path from $x$ to $u_i$ and the $T_y$ path from $y$ to $v_i$ do not share any colors. 
\begin{lemma}\label{lemcol}
Let $T_1,T_2$ be two vertex disjoint copies of an edge colored complete $d$-ary tree with $\ell$ levels, where 
$d\geq 3$. Let $T_1,T_2$ be rooted at $x,y$ respectively.
Suppose that the colorings
of $T_1,T_2$ are both rainbow. Let $\k=(d-1)^{\ell}$. Then there exist leaves $u_1,u_2,\ldots,u_\k$ of $T_1$ and 
leaves 
$v_1,v_2,\ldots v_\k$ of 
$T_2$ such that the following is true: If $P_i,P_i'$ are the paths from $x$ to $u_i$ in $T_1$ and from $y$ to $v_i$ in 
$T_2$ respectively, then 
$P_i\cup P_i'$ is rainbow colored for $i=1,2,\ldots,\k$.
\end{lemma}
\begin{proof}
Let $A_\ell$ be the minimum number of rainbow path pairs that we can find in any such pair of edge colored trees.
We prove that $A_\ell\geq (d-1)^\ell$ by induction on $\ell$. This is true trivially for $\ell=0$. 
Suppose that $x$ is incident with $x_1,x_2,\ldots,x_d$ and that the sub-tree rooted at $x_i$ is $T_{1,i}$ for 
$i=1,2,\ldots,d$.
Define $y_i$ and $T_{2,i},\,i=1,2,\ldots,d$ similarly with respect to $y$. Suppose that the color of the 
edge $(x,x_i)$ is 
$c_i$ for $i=1,2,\ldots,d$
and let $Q_x=\set{c_1,c_2,\ldots,c_d}$. 
Similarly, suppose that the color of the edge $(y,y_i)$ is 
$c_i'$ for $i=1,2,\ldots,d$
and let $Q_y=\set{c_1',c_2',\ldots,c_d'}$. 
Next suppose that $Q_j$ is the set of colors in $Q_x$ that appear on the 
edges $E(T_{2,j})\cup \{(y,y_j)\}$ .
The sets $Q_1,Q_2,\ldots,Q_d$ are pair-wise disjoint. 
Similarly, suppose that $Q_i'$ is the set of colors in $Q_y$ that appear on the 
edges $E(T_{1,i})\cup \{(x,x_i)\}$.
The sets $Q_1',Q_2',\ldots,Q_d'$ are pair-wise disjoint. 

Now define a bipartite graph $H$ with vertex set $A+B=[d]+[d]$ and an edge $(i,j)$ iff $c_i\notin Q_j$
and $c_j'\notin Q_i'$. We claim that if $S\subseteq A$ then its neighbor set $N_H(S)$ satisfies the inequality
\beq{HM}
d|S|-|N_H(S)|-|S|\leq |S|\cdot |N_H(S)|.
\eeq
Here the LHS of \eqref{HM} bounds from below, the size of the set $S:N_H(S)$ of edges between $S$ and $N_H(S)$. 
This is because there are at most 
$|S|$ edges missing from $S:N_H(S)$ due to $i\in S$ and $j\in N_H(S)$ and $c_i\in Q_j$. 
At most $|N_H(S)|$ edges are missing
for similar reasons. On the other hand, $d|S|$ is the number there would be without these missing edges. The RHS
of \eqref{HM} is a trivial upper bound.

Re-arranging we get that 
$$|N_H(S)|-|S|\geq \rdup{\frac{(d-2-|S|)|S|}{|S|+1}}\geq -1.$$
(We get -1 when $|S|=d$).

Thus $H$ contains a matching $M$ of size $d-1$. Suppose without loss of generality that this matching is 
$(i,i),i=1,2,\ldots,d-1$.
We know by induction that for each $i$ we can find paths $(P_{i,j},\hP_{i,j}),\,j=1,2,\ldots,(d-1)^{\ell-1}$ 
where $P_{i,j}$ is a root to leaf path 
in $T_{1,i}$ and $\hP_{i,j}$ is a root to leaf path 
in $T_{2,i}$  and that $P_{i,j}\cup \hP_{i,j}$ is rainbow for all $i,j$. Furthermore, $(i,i)$ being an edge of $H$, 
means that the edge sets $\set{(x,x_i)}\cup E(P_{i,j})\cup E(\hP_{i,j})\cup \{(y,y_i\}$ are all rainbow. 
\end{proof}

Let 
$$V_1=\set{x:V(T_x)\text{ contains a cycle}}.$$ 
When $x,y\notin V_1$ we apply this Lemma to $T_x,T_y$ by deleting one of the $r$ sub-trees attached to each of $x,y$
and applying the lemma directly to the $(r-1)$-ary trees that remain. This will yield $(r-2)^k$ pairs
of paths. 
If $x\in V_1$, we delete $r-2$ sub-trees attached to $x$ leaving at least two $(r-1)$-ary trees of depth $k-1$
with roots adjacent to $x$. We can do the same at $y$. Let $c_1,c_2$ be the colors of the two edges 
from $x$ to the roots of these two trees $T_1,T_2$. Similarly, let $c_1',c_2'$ be the colors of the two analogous 
edges from $y$ to the trees $T_1',T_2'$. If color $c_1$ does not appear in $T_1'$ then we apply
the lemma to $T_1$ and $T_1'$. Otherwise, we can apply the lemma to $T_1$ and $T_2'$. In both cases we 
obtain  $(r-2)^{k-1}$ pairs
of paths.  

Accounting for bad vertices we put 
$$\s=\s_0-6=(r-2)^{k-1}-6\geq \frac{\log n}{r-2}-6$$ 
and we see {\red from Remark \ref{rem2} 
that we can \whp\ find $\s$ paths $P_1,P_2,\ldots,P_\s$ 
of length $O(\log n)$ from $x$ to $y$. Path $P_i$ goes from $x$ to a leaf $u_i\in L_x^*$ via $T_x$ and then
traverses $Q_i=P(u_i,v_i)$ where {\rred $v_i=\f(u_i)\in L_y^*$} and then goes from $v_i$ to a $y$ via $T_y$. 
Here $\f$ is some partial map from $L_x^*$ to $L_y^*$. It is a random variable that depends on the coloring
$\cC$ of the edges of $T_x$ and $T_y$.
The paths $P_1,P_2,\ldots,P_\s$ depend 
on the choice of $\f$ and hence $\cC$ and so we should write $P_i=P_i(\cC)$.

We fix the coloring $\cC$ and hence $P_1,P_2,\ldots,P_\s$. Let $\cR$ be the event that at least one of the paths
$P_1,P_2,\ldots, P_\s$ is rainbow colored. 
We show that $\Pr(\neg\cR\mid\cC)$ is small.

We let $c(e)$ denote the color of edge $e$ in a given coloring. 
We remark next that for a particular coloring $c_1,c_2,\ldots,c_m$ of the edges $e_1,e_2,\ldots,e_m$ we have
$$\Pr(c(e_i)=c_i,\,i=1,2,\ldots,m)=\prod_{i=1}^m\frac{1}{a_i}$$
where $q-\D\leq a_i\leq q$ is the number of colors available for the color of the edge $e_i$ 
given the coloring so far i.e. the number of colors
unused by the neighbors of $e_i$ in $\G$ when it is about to be colored.

Now fix an edge $e=e_i$ and the colors $c_j,\,j\neq i$. Let $C$ be the set of colors not used by the neighbors of $e_i$ in $\G$.
The choice by $e_i$ of its color under this conditioning is not quite random, but close. Indeed, we claim that for $c,c'\in C$
$$\frac{\Pr(c(e)=c\mid c(e_j)=c_j,\,j\neq i)}{\Pr(c(e)=c'\mid c(e_j)=c_j,\,j\neq i)}\leq \bfrac{q-\D}{q-\D-1}^\D.$$
This is because, changing the color of $e_i$ only affects the number of colors available to neighbors of $e_i$, and only by at most one.

Thus, for $c\in C$, we have
$$\Pr(c(e)=c\mid c(e_j)=c_j,\,j\neq i)\leq \frac{1}{q-\D}\bfrac{q-\D}{q-\D-1}^\D.$$
Now $\D\leq (r-1)^{2k}=q/10$ and we deduce that 
$$\Pr(c(e)=c\mid c(e_j)=c_j,\,j\neq i)\leq \frac{2}{q}.$$
It follows that for $i\in[\s]$,
$$\Pr(P_i\text{ is rainbow colored}\mid \cC,\text{ coloring of }\bigcup_{j\neq i}Q_j)\geq 
\brac{1-\frac{4(k+\g)}{q}}^{2\g}.$$
This is because when we consider the coloring of $Q_i$ there will always be at most $2k+2\g$ colors forbidden by 
non-neighboring edges, if it is to be rainbow colored.

It then follows that
\begin{align*}
\Pr(\neg\cR\mid \cC)&\leq \brac{1-\brac{1-\frac{4(k+\g)}{q}}^{2\g}}^{\s}\\
&\leq \bfrac{8\g(k+\g)}{q}^{\s}\\
&\leq \bfrac{(2+10\e)\log_{r-1}^2n}{10\log^{\th_r}n}^{\s}=o(n^{-2}).
\end{align*}
This completes the proof of Theorem \ref{thrm:regular} when $r\geq 4$.

\noindent
{\bf Case 2: $r=3$:}\\
When $r=3$ we can't use $(r-2)^k$ to any effect. Also, we need to increase $q$ to $\log^4n$.
This necessary for a variety of reasons. One reason is that we will reduce $\s$ to $2^{k/2}$.
We want this to be $\Omega(\log n)$ and 
this will force $k$ to (roughly) double what it would have been if we had followed the recipe for $r\geq 4$.
This makes $\D$ close to $\log^4n$ and we need $q\gg\D$. 

And we need to modify the 
argument based on Lemma \ref{lemcol}. Instead of inducting on the trees at depth one from the roots $x,y$, 
we now 
induct on the trees at depth two. Assume first that $x,y\notin V_1$.
After ignoring one branch for $T_x$ and $T_y$ we now consider the sub-trees $T_{x,i},T_{y,i},\,i=1,2,3,4$ of 
$T_x,T_y$ whose
roots $x_1,\ldots,x_4$ and $y_1,\ldots,y_4$ are at depth two. We cannot necessarily make this construction when 
$x\in V_1$.
Let $P_i$ be the path from $x$ to $x_i$ in $T_x$ and let $\hP_j$ be the path from $y$ to $y_j$ in $T_y$.
Next suppose that $\widehat{Q}_j$ is the set of colors in $Q$ that appear on the 
edges $E(T_{y,j})\cup E(\hP_j)$.
Similarly, suppose that $Q_i'$ is the set of colors in $Q'$ that appear on the 
edges $\{E(T_{x,i})\cup E(P_i)\}$. 

Re-define $H$ to be the bipartite graph with vertex set $A+B=[4]+[4]$. The edges of $H$ are as before:
$(i,j)$ exists iff $c_i\notin Q_j$ and $c_j'\notin \widehat{Q}_i$. This time we can only say that a color is in at 
most 
two $\widehat{Q}_i$'s 
and
similarly for the $Q_j'$'s.
The effect of this is to replace \eqref{HM} by
$$4|S|-2(|N_H(S)|+|S|)\leq |S|\cdot |N_H(S)|$$
from which we can deduce that
$$|S|-|N_H(S)|\leq \frac{|S|\cdot |N_H(S)|}{2}\leq 2|N_H(S)|.$$
It follows that $|N_H(S)|\geq \rdup{|S|/3}\geq |S|-2$ and so $H$ contains a matching of size two. 
An inductive argument then shows that we are able to find $2^{\rdown{k/2}}$ 
rainbow pairs of paths. The proof now continues as in the case $r\geq 4$, arguing about the coloring of
paths $P_1,P_2,\ldots,P_\s$ where now $\s=2^{\rdown{k/2}}$. 

We finally deal with the vertices in $V_1$. 
We classify them according to the size of the cycle $C_x$ that is contained in $V(T_x)$.
If $T_x$ contains a cycle $C_x$ then necessarily 
$|C_x|\leq 2k$ and so there are at most $2k$ types in our classification. 
It follows from Lemma \ref{density} that if $x,y\in V_1$ and $T_x\cap T_y\neq\emptyset$
then $C_x=C_y$ \whp. Note next that the distance from $x$ to $C_x$ is at most $k-|C_x|/2$. If $C$ is a cycle
of length at most $2k$, let $V_C=\set{x:C=C_x}$ and let $E_C$ be the set of edges contained in $V_C$. We have
\beq{VK}
|V_C|=O( |C|2^{k-|C|/2})=O(2^k)=O(\log^2n/\log\log n).
\eeq
We introduce $2k$ new sets $\hC_i,i=3,4,\ldots,2k$ of
$O(\log^2n/\log\log n)$ colors, distinct from $Q$. Thus we introduce $O(\log^2n)$ new colors overall.
We re-color each $E_C$ with the colors from $\hC_{|C|}$. It is important to observe that if $|C|=|C'|$ then the graphs 
induced by
$V_C$ and $V_{C'}$ are isomorphic and so we can color them isomorphically. By the latter we mean that we choose some 
isomorphism $f$ from $V_C$ to $V_{C'}$ and then if $e$ is an edge
of $V_C$ then we color $e$ and $f(e)$ with the same color. After this re-coloring, we see that if $T_x$ and $T_y$ are 
not vertex disjoint,
then they are contained in the same $V_C$. The edges of $V_C$ are rainbow colored and so now we only need to concern 
ourselves with
$x,y\in V_1$ such that $T_x$ and $T_y$ are vertex disjoint. Assume now that $x,y\in V_1$.

Assume first that $x,y$ are of the same type and that they are at the same distance from $C_x,C_y$ respectively.
Our aim now is to define binary trees $T_x',T_y'$ ``contained`` 
in $T_x,T_y$ that can be used as in Lemma \ref{lemcol}. 
If we delete an edge $e=(u,v)$ of $C_x$ then the graph that remains on $V(T_x)$
is a tree with at most two vertices $u,v$ of degree two. Now delete one of the three sub-trees of $T_x$.
If there are vertices of degree two, make sure one of them is in this sub-tree. 
If necessary, shrink the path of length two with the remaining vertex of degree two in the middle to an edge $e_x$. 
It has leaves at depth $k-1$ and leaves at depth $k-2$. 
The resulting binary tree will be our $T_x'$. The leaves at depth $k-1$ come in pairs. Delete one vertex from each 
pair
and shrink the paths of length two through the vertex at depth $k-2$ to an edge.

The edges that are obtained by shrinking paths of length two will have two colors. 
Because $x,y$ are at the same distance from their cycles, we can delete $f(e)$ from $C_y$ and do the construction so 
that
$T_x'$ and $T_y'$ will be isomorphically colored.

It is now easy to find $2^{k-2}$ pairs of paths whose unions are rainbow colored. Each leaf of $T_x,T_y$
can be labelled by a $\{0,1\}$ string of length $k-2$. We pair string $\xi_1\xi_2\cdots\xi_{k-1}\xi_{k-2}$ in $T_x$
with $(1-\xi_1)\xi_2\cdots\xi_{k-1}\xi_{k-2}$ in $T_y$. The associated paths will have a rainbow union.
The proof now continues as in the case $r\geq 4$, arguing about the coloring of
paths $P_1,P_2,\ldots,P_\s$ where now $\s=2^{k-2}$.

If $x$ is further from $C_x$ than $y$ is from $C_y$ then let $z$ be the vertex on the path from $x$ to $C_x$ at the
same distance from $C_x$ as $y$ is from $C_y$. We have a rainbow path from $z$ to $y$ and adding the $T_x$ path 
from $x$ to $z$
gives us a rainbow path from $x$ to $y$. This relies on the fact that $V_{C_x}$ and $V_{C_y}$ are isomorphically
colored.

If $x,y$ are of a different type, then $T_x$ and $T_y$ are re-colored with distinct colors and we can proceed as 
as in the case $r\geq 4$, arguing about the coloring of
paths $P_1,P_2,\ldots,P_\s$ where now $\s=2^{k}$, using Corollary \ref{nlem1}.

If $x\in V_1$ and $y\notin V_1$ then we can proceed as if 
both are not in $V_1$.
This is because of the re-coloring of the edges of $T_x$. We can proceed as 
as in the case $r\geq 4$, arguing about the coloring of
paths $P_1,P_2,\ldots,P_\s$ where now $\s=2^{k}$, using Corollary \ref{nlem1}.

This completes our proof of 
Theorem \ref{thrm:regular}.
\qed
\section{Conclusion} 
\label{sec:concl} 

\noindent In this work we have given an aymptotically tight result on 
the rainbow connectivity of $G=G(n,p)$ at the connectivity
threshold. It is reasonable to conjecture that this could be tightened:

{\bf Conjecture:} {\em Whp}, $rc(G) = \max\set{Z_1,diameter(G(n,p))}$.

\noindent Our result on random regular graphs is not so tight. 
It is still reasonable to believe that the above conjecture also holds in
this case. (Of course $Z_1=0$ here).

\noindent It is worth mentioning that if the degree $r$ in Theorem \ref{thrm:regular} is allowed to grow as fast 
as $\log n$ then
one can prove a result closer to that of Theorem \ref{thrm:mainthrm}.

\end{document}